\documentclass[12pt, reqno]{amsart}
\usepackage{amsmath, amsthm, amsfonts, amssymb, color}
 \usepackage{mathrsfs}
\usepackage{amsfonts, amsmath}
\usepackage{amsmath,amstext,amsthm,amssymb,amsxtra}
 \usepackage{txfonts} %also pxfonts
\usepackage[%colorlinks, 
citecolor=blue,pagebackref,hypertexnames=false]{hyperref}
 \allowdisplaybreaks
 \usepackage{pgf,tikz}
\begin{document}
\baselineskip 16pt

	\newtheorem{theorem}{Theorem}[section]
	\newtheorem{proposition}{Proposition}[section]
	\newtheorem{coro}{Corollary}[section]
	\newtheorem{lemma}{Lemma}[section]
	\newtheorem{definition}{Definition}[section]
	\newtheorem{assum}{Assumption}[section]
	\newtheorem{example}{Example}[section]
	\newtheorem{remark}{Remark}[section]
	\newcommand\Kone{K^{j,k}}
    \newcommand\Ktwo{\mathcal{K}^{j,k}}
    \newcommand\KKone{{\mathbb K}^{j,k}}
    \newcommand\KKtwo{\mathcal{B}^{j,k}}
	\newcommand\R{\mathbb{R}}
	\newcommand\RR{\mathbb{R}}
	\newcommand\CC{\mathbb{C}}
	\newcommand\NN{\mathbb{N}}
	\newcommand\ZZ{\mathbb{Z}}
	\def\RN {\mathbb{R}^n}

	\newcommand{\mc}{\mathcal}
	\newcommand\D{\mathcal{D}}	
	\newcommand{\supp}{{\rm supp}{\hspace{.05cm}}}
	\newcommand {\Rn}{{\mathbb{R}^{n}}}
	\newcommand {\rb}{\rangle}
	\newcommand {\lb}{{\langle}}
	\newtheorem{corollary}[theorem]{Corollary}
	\numberwithin{equation}{section}

\bibliographystyle{amsplain}

\title[Characterizations of harmonic quasiregular mappings]
{Characterizations of harmonic quasiregular\\ mappings in function spaces}


\thanks{$^*$Corresponding author.} 
\author[J. Sun, J. Liu and Z.-G. Wang]{Jihua Sun, Junming Liu and Zhi-Gang Wang$^*$}

\address{\noindent Jihua Sun\vskip.01in
		School of Mathematics and Statistics, Guangdong University of Technology, Guangzhou 510520, Guangdong,  P.~R.~China.}
	\email{\textcolor[rgb]{0.00,0.00,0.84}{2112414014$@$mail2.gdut.edu.cn}}

\address{\noindent Junming Liu\vskip.01in
		School of Mathematics and Statistics, Guangdong University of Technology, Guangzhou 510520, Guangdong,  P.~R.~China.}
	\email{\textcolor[rgb]{0.00,0.00,0.84}{jmliu$@$gdut.edu.cn}}

\address{\noindent Zhi-Gang Wang\vskip.01in
		School of Mathematics and Statistics, Hunan First Normal University, Changsha 410205, Hunan, P. R. China.}
	\email{\textcolor[rgb]{0.00,0.00,0.84}{wangmath$@$163.com}}

\date{\today}

\begin{abstract}
We study conjugate-type phenomena for complex-valued harmonic quasiregular mappings in the unit disk across three function space families: $Q(n,p,\alpha)$, $F(p,q,s)$, and the non-derivative $M(p,q,s)$. For a harmonic $K$-quasiregular mapping $f=u+iv$, we first show that if the real part $u$ belongs to $Q_h(1,p,\alpha)$ (with $\alpha>-1$ and $\alpha+1<p<\alpha+2$), the imaginary part $v$ lies in the same space with a $K$-dependent quantitative bound. An analogous stability result is established for the harmonic $F$-scale, with sharp $K$-dependence. These results are extended to harmonic $(K, K')$-quasiregular mappings, yielding explicit estimates with an additional inhomogeneous term involving $K'$. Finally, for normalized harmonic quasiconformal mappings, %$f\in\mathcal S_H(K)$, 
we derive membership criteria in the harmonic $M$- and $F$-scales, and obtain corresponding conclusions for their natural derivatives, with parameter ranges governed by the order $\alpha_K$ of the family of harmonic quasiconformal mappings.
\end{abstract}

\maketitle
\pagestyle{myheadings}

\tableofcontents	

\section{Introduction}
Let $\mathbb{D}:=\{z \in \mathbb{C}: |z| < 1\}$ denote the open unit disk in the complex plane $\mathbb{C}$, and let  
$\partial\mathbb{D}:=\{z\in\mathbb{C}: |z| = 1\}$ denote the unit circle. We denote by $H(\mathbb{D})$ the space of all analytic functions in $\mathbb{D}$, endowed with the topology of uniform convergence on compact subsets. For $f\in H(\mathbb{D})$ and $z=b e^{i\theta}$ with $b=|z|\in[0,1)$, by noting that $f(be^{i\theta})$ is independent of $\theta$ when $b=0$, it is natural to set
$f_\theta(0)=0$.
Furthermore, we adopt the convention $(bf_b)(0)=0$, ensuring that both $f_\theta$ and $bf_b$
are well-defined on $\mathbb D$.

For $z=x+iy$ and a complex-valued function $f$ on a domain $\Omega\subset\mathbb C$,
we write the Wirtinger derivatives as
\[
f_z=\frac12\left(f_x-i f_y\right),\quad
f_{\bar z}=\frac12\left(f_x+i f_y\right),
\]
whenever the partial derivatives exist. We refer the reader to \cite[Chapters~1-2]{ABR} for background on harmonic function theory and related notations.

For $\theta\in[0,2\pi)$, the directional derivative of $f$ at $z$ is defined by
\[
\partial_\theta f(z)
= f_z(z)e^{i\theta}+f_{\bar z}(z)e^{-i\theta}.
\]
Accordingly, we set
\[
\Lambda_f(z):=\max_{\theta\in[0,2\pi)}|\partial_\theta f(z)|
=|f_z(z)|+|f_{\bar z}(z)|
\]
and
\[\lambda_f(z):=\min_{\theta\in[0,2\pi)}|\partial_\theta f(z)|
=\bigl||f_z(z)|-|f_{\bar z}(z)|\bigr|.\]

For $K\ge 1$, a sense-preserving harmonic function $f$ is said to be $K$-quasiregular if the
dilatation
\[
D_f:=\frac{|f_{z}|+|f_{\bar z}|}{|f_{z}|-|f_{\bar z}|}\le K
\]
throughout $\mathbb{D}$. Writing $f=h+\overline{g}$, this condition is equivalent to the requirement that its analytic dilatation $w:={g'}/{h'}$ satisfies the inequality
\[
|w(z)|\le k<1 \quad (z\in\mathbb{D}),
\]
where
\begin{equation}\label{kandK}
k:=\frac{K-1}{K+1}.
\end{equation}
The function $f$ is called $K$-quasiconformal if it is $K$-quasiregular and homeomorphic in
$\mathbb{D}$.

A mapping $f:\Omega\to\mathbb C$ is said to be absolutely continuous on lines (ACL) in $\Omega$
if $f$ is absolutely continuous on almost every horizontal line segment and almost every vertical
line segment contained in $\Omega$. We write $f\in {\rm ACL}^2(\Omega)$ if $f\in {\rm ACL}(\Omega)$ and
$f_x,f_y\in L^2_{\rm loc}(\Omega)$.

The Jacobian of $f$ is defined by
\[
J_f = |f_z|^2-|f_{\bar z}|^2 = \Lambda_f\,\lambda_f.
\]
Following the standard definition, we say that $f$ is a {$(K,K')$-quasiregular mapping} in $\Omega$
if $f\in {\rm ACL}^2(\Omega)$, $J_f\ge 0$ a.e. in $\Omega$, and
\[
\Lambda^2_f(z)\le K\,J_f(z)+K'
\] 
for a.e. $z\in\Omega$,
where $K\ge 1$ and $K'\ge 0$ are constants.
In particular, when $K'=0$, $f$ reduces to the classic {$K$-quasiregular} mapping.
%Moreover, $f$ is called \emph{$K$-quasiconformal} in $\Omega$ if it is $K$-quasiregular and
%homeomorphic in $\Omega$.

If, in addition, $f$ is harmonic in $\mathbb D$, then there exist analytic functions $h,g\in H(\mathbb D)$
such that
\[
f=h+\overline g,
\]
and (with the normalization $g(0)=0$ if needed) this decomposition is unique up to an additive constant.
In this case, we have
\[
f_z=h'(z),\quad f_{\bar z}=\overline{g'(z)},
\]
and the (second) complex dilatation is
\[
w(z):=\frac{g'(z)}{h'(z)}\quad (h'(z)\neq 0).
\]
The mapping is sense-preserving in $\mathbb D$ if and only if $J_f>0$, %in $\mathbb D$
or equivalently, $|w(z)|<1$ in $\mathbb D$.

%We also use the standard normalized class of univalent harmonic mappings.
Let $\mathcal S_H$ denote the family of all sense-preserving univalent harmonic mappings
$f=h+\overline g$ in $\mathbb D$ satisfying the normalizations
$h(0)=h'(0)-1=g(0)=0.$

For $K\ge 1$, we set
\[
\mathcal S_H(K):=\left\{f\in\mathcal S_H: f \text{ is } K\text{-quasiconformal in }\mathbb D\right\}.
\]
The \emph{order} of $\mathcal S_H(K)$ is defined by
\[
\alpha_K:=\frac{1}{2}\sup_{f\in\mathcal S_H(K)}|h''(0)|.
\] 
We note that $\alpha_K$ is conjectured to be ${(3K+1)}/{(K+1)}$ in Conjecture 2 of \cite{wwrq},
For more investigations on harmonic quasiconformal mappings, see e.g., \cite{cp,p,LZ,whlk,wsj}.

Wang-Wang-Rasila-Qiu \cite{wwrq} (see also \cite{wqr}) constructed the harmonic quasiconformal Koebe function, which provides novel perspectives and tools for the study of the theory of harmonic quasiconformal mappings. Li-Ponnusamy \cite{lp} have verified that the harmonic quasiconformal Koebe function is indeed extremal for several subclasses with geometric properties of harmonic quasiconformal mappings. Furthermore, Das-Huang-Rasila \cite{dhr} studied the Hardy spaces of convex and close-to-convex harmonic quasiconformal mappings of the class $\mathcal S_H(K)$, Das-Rasila \cite{DR} obtained the order of the class $\mathcal S_H(K)$ belongs to harmonic Bergman spaces.

Problems of conjugate type originate from classical results for analytic functions.
Regarding the Hardy spaces, the celebrated theorem of Riesz asserts that harmonic conjugation is bounded on $H^p$
for $1<p<\infty$, see e.g., \cite{DP1}.
At the endpoints, one needs finer substitutes (weak-type and BMO-type replacements),
and the real-variable $H^p$ framework provides a systematic language for such phenomena (cf.\ \cite{DP1,FS}).
In contrast, Hardy-Littlewood showed that if $f=u+iv$ is analytic in $\mathbb D$, then
membership of $u$ in the (analytic) Bergman space $A^p$ forces $v$ to belong to the same space for every
$0<p<\infty$, see \cite{HL} (and also the standard accounts \cite{DS,ZK}).
In other words, Bergman-type integrability behaves markedly better than Hardy-type integrability
with respect to taking conjugates.

For general complex-valued harmonic mappings, these analytic conclusions do not carry over verbatim.
A key insight in the recent literature is that geometric control of the mapping restores
conjugate-type principles in the harmonic setting:
quasiconformality/quasiregularity turns out to be a natural sufficient hypothesis for Riesz-, Kolmogorov-, and Zygmund-type statements for harmonic mappings, see e.g., \cite{AK,CHWX,CK,DHR,KD1,KD2}.
Here we follow \cite{CK} and say that $f$ is {harmonic $(K,K')$-quasiregular} in a domain $\Omega$
if $f$ is harmonic in $\Omega$ and $(K,K')$-quasiregular in $\Omega$, where $K\ge1$ and $K'\ge0$ are constants.
%In particular, when $K'=0$ this reduces to the usual notion of a \emph{harmonic $K$-quasiregular} mapping,
%and if in addition $f$ is a homeomorphism, then $f$ is called \emph{harmonic $K$-quasiconformal} in $\Omega$.

Within this direction, several recent contributions sharpen and extend the classical paradigm.
For instance, Das-Rasila \cite{DR} proved that the Hardy-Littlewood implication on Bergman spaces remains valid
for harmonic $K$-quasiregular mappings in $\mathbb D$, and they further obtained Bergman-type integrability
for natural derivatives of mappings in $\mathcal S_H(K)$ (see \cite{DRN,LZ}).
Related quantitative estimates for univalent (or locally univalent) harmonic mappings are developed in \cite{DSK},
while boundary behaviour for harmonic univalent maps goes back at least to \cite{AML}.
Very recently, Chen-Huang-Wang-Xiao \cite{CHWX} and Chen-Kalaj \cite{CK} investigated sharp conjugate-type
inequalities for quasiregular and harmonic $(K,K')$-quasiregular mappings, including Riesz-, Kolmogorov-
and Zygmund-type theorems and sharp (or asymptotically sharp) constants for model classes.

Motivated essentially by the above developments, we study conjugate-type properties in a unified scale of Möbius-invariant
function spaces, notably the families $F(p,q,s),~Q(n,p,\alpha)$ and its prominent special cases such as Morrey spaces $\mathcal{L}^{2,\lambda}$,
Bergman--Morrey spaces $A^{p,\lambda}$, and $Q_s$-spaces, see e.g., \cite{LL,WJ,XJ1,XJ2,YL,ZR,ZR1,ZK,ZK1,z2}.

%\textbf{Organization of the paper.}
This paper is organized as follows. Section~\ref{section02} provides some preliminary results.
Section~\ref{section2} is devoted to conjugate-type stability for harmonic quasiregular mappings in two derivative-based scales, namely, the families $Q_h(1,p,\alpha)$ and $F_h(p,q,s)$.
We first prove that for a harmonic $K$-quasiregular mapping $f=u+iv$, the condition $u\in Q_h(1,p,\alpha)$ (with $\alpha>-1$ and $\alpha+1<p<\alpha+2$) implies $v\in Q_h(1,p,\alpha)$ together with the sharp estimate in Theorem~\ref{v<u}.
We then establish the corresponding result in the $F$-scale (Theorem~\ref{v<u-F}), and record several standard specializations as corollaries, see Corollaries~\ref{coro:morrey-K}-\ref{coro:Qs-K}.
Finally, we extend both stability principles to harmonic $(K,K')$-quasiregular mappings and obtain explicit norm bounds with an additional term depending on $K'$, see Theorems~\ref{thm:Qh-CK21} and \ref{thm:Fh-KK'} and their consequences in Corollaries~\ref{coro:morrey-KK'}-\ref{coro:Qs-KK'}.
Section~\ref{section3} concerns normalized univalent harmonic mappings in $\mathcal S_H(K)$ and focuses on integrability consequences in the non-derivative scale $M_h(p,q,s)$ and the derivative-based scale $F_h(p,q,s)$.
We show that each $f\in\mathcal S_H(K)$ belongs to $M_h(p,q,s)$ in a full parameter range determined by the order $\alpha_K$, and we also obtain membership results for the derivatives $f_z$, $\overline{f_{\bar z}}$, $f_\theta$, and $b f_b$, see Theorem~\ref{M-th3}.
The corresponding assertions in the $F_h$-scale are established in Theorem~\ref{F-th3}. The arguments in Section~\ref{section3} are partially inspired by, and adapted from,
the methods developed by Das and Rasila in \cite{DR} for harmonic quasiregular mappings in Bergman spaces,
and we include this remark for proper attribution.

To avoid ambiguity, throughout this paper, we use $C$ to denote a positive constant, which may change from line to line. We write $A \lesssim B$ if $A \leq C B$, and $A \thickapprox B$ if both $A \lesssim B$ and $B \lesssim A$ hold.

\section{Preliminaries}\label{section02}

In order to prove our main results, we need the following definitions and lemmas.

\begin{definition}\label{def:Qnpa}
{\rm Let $0<p<\infty$, $-1<\alpha<\infty$, and let $n\in\NN$ be a positive integer. Denote by $Q(n,p,\alpha)$ the space of all analytic functions $f$ in $\mathbb D$ such that
\[
\|f\|^{p}_{n,p,\alpha}
:=\sup_{a\in\mathbb D}\int_{\mathbb D}\bigl|(f\circ \sigma_a)^{(n)}(z)\bigr|^{p}\,(1-|z|^{2})^{\alpha}\,dA(z)<\infty,
\]
where $$\sigma_a(z)=\dfrac{a-z}{1-\overline{a}z}$$ is a M\"obius automorphism of $\mathbb D$ and $dA$ is the area measure on $\mathbb D$.}
\end{definition}

The space $Q(n,p,\alpha)$ is M\"obius invariant in the sense that an analytic function $f$ belongs to $Q(n,p,\alpha)$ if and only if $f\circ\sigma$ belongs to $Q(n,p,\alpha)$ for every $\sigma\in {\rm Aut}(\mathbb D)$, moreover,
\[
\|f\circ\sigma\|_{n,p,\alpha}=\|f\|_{n,p,\alpha},\quad f\in Q(n,p,\alpha),\quad \sigma\in {\rm Aut}(\mathbb D).
\]
When $p\ge1$, $|f(0)|+\|f\|_{n,p,\alpha}$ defines a complete norm on $Q(n,p,\alpha)$, hence $Q(n,p,\alpha)$ is a Banach space; when $0<p<1$, $Q(n,p,\alpha)$ is not necessarily Banach but it is always complete as a metric space. Furthermore, $Q(n,p,\alpha)$ is trivial when $np>\alpha+2$, it contains all polynomials when $np\le \alpha+2$, and it coincides with the Besov space $B_p$ when $np=\alpha+2$; in addition, $Q(n,p,\alpha)\subset\mathcal{B}$ (the Bloch space). As special identifications, if $np<\alpha+1$, then $Q(n,p,\alpha)=\mathcal{B}$, while if $p=2$ and $\alpha=2n-1$, then $Q(n,p,\alpha)=\mathrm{BMOA}$. In the special case $n=1$ and $p=2$, one may apply a change of variables to rewrite
\[
\int_{\mathbb D}\bigl|(f\circ\sigma_a)'(z)\bigr|^{2}(1-|z|^{2})^{\alpha}\,dA(z)
\]
in the equivalent form
\[
\int_{\mathbb D}|f'(z)|^{2}\,(1-|\sigma_a(z)|^{2})^{\alpha}\,dA(z),
\]
so that the corresponding spaces $Q(n,p,\alpha)$ reduce to the classical $Q_{\alpha}$ scale. More generally, if $2n<\alpha+3$, then
$Q(n,2,\alpha)=Q_{\beta},~\beta=\alpha-2(n-1).$
Standard references \cite{XJ1} provide further background on the theory of $Q_{\alpha}$ spaces.

\begin{definition} {\rm(\cite{ZR})
 Let  $p, q$  and  $s$  be real constants such that  $0<p<\infty,-2<q<\infty$  and  $0<s<\infty.$ We say that a function  $f \in H(\mathbb{D})$  belongs to the space  $F(p, q, s)$  if
$$\|f\|_{p, q, s}^{p}=\sup _{a \in \mathbb{D}} \int_{\mathbb{D}}\left|f^{\prime}(z)\right|^{p}\left(1-|z|^{2}\right)^{q} g^{s}(z, a) \mathrm{d} A(z)<\infty,$$
where  $$g(z, a)=-\log \left|\sigma_{a}(z)\right|=\log \left|\frac{1-\bar{a} z}{a-z}\right|$$  is the Green's function of $\mathbb{D}$  with logarithmic singularity at $a \in \mathbb{D}.$}    
\end{definition}

Originating in Zhao’s work \cite{ZR1}, the three-parameter family $F(p,q,s)$ provides a unified framework that encompasses a host of classical analytic spaces-Hardy, Bergman, Dirichlet, Besov, Bloch and $BMOA$ among them-allowing one to move seamlessly between them by tuning $(p,q,s).$ Moreover, the following characterizations of $F(p,q,s)$ using Carleson measure are from \cite[Theorem 3.3]{ZR1}:

\begin{lemma}
Let  $0<p<\infty,-2<q<\infty, 0< s<\infty$  satisfy  $q+s>-1,$ let  $f$  be analytic on  $\mathbb{D},$ and let  $$d\mu_{f}(z)=\left|f^{\prime}(z)\right|^{p}\left(1-|z|^{2}\right)^{q+s} dA(z).$$ Then the following statements are equivalent:

$(i) ~~f \in F(p, q, s) ;$

$(ii) ~\sup _{a \in \mathbb{D}} \int_{\mathbb{D}}\left|f^{\prime}(z)\right|^{p}\left(1-|z|^{2}\right)^{q}\left(1-\left|\sigma_{a}(z)\right|^{2}\right)^{s} \mathrm{~d} A(z)<\infty ;$

$(iii)~~d\mu_{f} \text{is a bounded } $s-$\text{Carleson measure.}$

\noindent Moreover, we have
$$\|f\|_{p, q, s}^{p} \approx \sup _{a \in \mathbb{D}} \int_{\mathbb{D}}\left|f^{\prime}(z)\right|^{p}\left(1-|z|^{2}\right)^{q}\left(1-\left|\sigma_{a}(z)\right|^{2}\right)^{s} dA(z).$$ 
\end{lemma}

This scheme specializes to many familiar identifications, for example $F(2,0,1)=BMOA.$ Thereby highlighting the role of $F(p,q,s)$ as a common umbrella for Möbius-invariant and weighted theories.

\begin{definition}
{\rm For  $-2<q<\infty,~0<s<\infty,~q+s>-1.$ associated with the $F(p,q,s)$ spaces, the spaces  $M(p,q,s)$ consist of $f\in H(\mathbb{D})$ such that
$$\|f\|_{M(p,q,s)}=|f(0)|+\left(\sup _{a \in \mathbb{D}} \int_{\mathbb{D}}|f(z)|^{p}\left(1-|z|^{2}\right)^{q}\left(1-\left|\sigma_{a}(z)\right|^{2}\right)^{s} d A(z)\right)^{\frac{1}{p}}<\infty.$$}
\end{definition}

Clearly, when $p\ge 1,$ $M(p,q,s)$  is a Banach space with the above norm.

\begin{definition}\label{def:Qh-npa}
{\rm Let $0<p<\infty$, $-1<\alpha<\infty$, and let $n\in\NN$ be a positive integer.
A complex-valued harmonic function $f$ in $\mathbb D$ is said to belong to the harmonic space $Q_h(n,p,\alpha)$ if
\[
\|f\|_{Q_h(n,p,\alpha)}^{p}
:=\sup_{a\in\mathbb D}\int_{\mathbb D}\Bigl(\bigl|(h\circ\sigma_a)^{(n)}(z)\bigr|+\bigl|(g\circ\sigma_a)^{(n)}(z)\bigr|\Bigr)^{p}
(1-|z|^{2})^{\alpha}\,dA(z)<\infty,
\]
where $f=h+\overline g$ is the canonical decomposition with $h,g\in H(\mathbb D)$ and $g(0)=0.$

In particular, when $n=1$, one may write equivalently
\begin{align*}
\|f\|_{Q_h(1,p,\alpha)}^{p}
&=\sup_{a\in\mathbb D}\int_{\mathbb D}\Lambda_{\,f\circ\sigma_a}(z)^{p}(1-|z|^{2})^{\alpha}\,dA(z)\\
&\approx\sup_{a\in\mathbb D}\int_{\mathbb D}\bigl|\nabla(f\circ\sigma_a)(z)\bigr|^{p}(1-|z|^{2})^{\alpha}\,dA(z),
\end{align*}
where $$\Lambda_f(z):=|f_z(z)|+|f_{\bar z}(z)|$$ and $$|\nabla f(z)|:=\sqrt{|f_x(z)|^2+|f_y(z)|^2},$$
and the pointwise comparability $$\Lambda_f(z)\le |\nabla f(z)|\leq\sqrt2\Lambda_f(z)$$ holds for all $z\in\mathbb D$.}
\end{definition}

\begin{definition}\label{def:Mh-pqs}
{\rm Let $0<p<\infty$, $-2<q<\infty$, $0<s<\infty$, and $q+s>-1$.
A complex-valued harmonic function $f$ in $\mathbb D$ is said to belong to the harmonic space $M_h(p,q,s)$ if
\[
\|f\|_{M_h(p,q,s)}
:=|f(0)|+\left(\sup_{a\in\mathbb D}\int_{\mathbb D}|f(z)|^{p}(1-|z|^{2})^{q}\bigl(1-|\sigma_a(z)|^{2}\bigr)^{s}\,dA(z)\right)^{\frac1p}<\infty.
\]}
\end{definition}

\begin{definition}\label{def:Fh-pqs}
{\rm Let $0<p<\infty$, $-2<q<\infty$, $0<s<\infty$, and $q+s>-1$.
A complex-valued harmonic function $f$ in $\mathbb D$ is said to belong to the harmonic space $F_h(p,q,s)$ if
\[
\|f\|_{F_h(p,q,s)}^{p}
:=\sup_{a\in\mathbb D}\int_{\mathbb D}\Lambda_f(z)^{p}(1-|z|^{2})^{q}\,g(z,a)^{s}\,dA(z)<\infty.
\]
Equivalently, one may replace $\Lambda_f$ by the (real) gradient size
\[
|\nabla f(z)|:=\sqrt{|f_x(z)|^2+|f_y(z)|^2},
\]
since $$\Lambda_f(z)\leq|\nabla f(z)|\leq\sqrt2\Lambda_f(z)$$ for all $z\in\mathbb D$. Moreover, one also has the equivalent semi-norm
\[
\|f\|_{F_h(p,q,s)}^{p}\approx
\sup_{a\in\mathbb D}\int_{\mathbb D}\Lambda_f(z)^{p}(1-|z|^{2})^{q}\bigl(1-|\sigma_a(z)|^{2}\bigr)^{s}\,dA(z).
\]}
\end{definition}

\begin{definition}
{\rm Let $\alpha > 0$. The Bloch-type space $\mathcal{B}_\alpha$ consists of all functions $f \in H(\mathbb{D})$ such that
\[\|f\|_{\mathcal{B}_\alpha} := |f(0)| + \sup_{z \in \mathbb{D}} \left(1-|z|^2\right)^\alpha |f'(z)|.\]}
\end{definition}
When $\alpha = 1$, we obtain the classical Bloch space $\mathcal{B} := \mathcal{B}_1$. For more details on Bloch-type spaces, see~\cite{ZK1}.

\begin{definition} {\rm (\cite{XJ2})
For $s>0,$ a function $f$ in $H(\mathbb{D})$ is said to belong to $Q_{s}$ space if
$$\|f\|_{Q_s}^2 =|f(0)|^2+ \sup _{a \in \mathbb{D}} \int_{\mathbb{D}}\left|f^{\prime}(z)\right|^{2}\left(1-\left|\sigma_{a}(z)\right|^{2}\right)^{s} dA(z).$$ }
\end{definition}

\begin{definition} {\rm (\cite{LL})
	For $0<\lambda\leq1$, a function $f$ in $H(\mathbb{D})$ is said to belong to Morrey space $\mathcal{L}^{2,\lambda}$ if
	$$\|f\|_{\mathcal{L}^{2,\lambda}}:=|f(0)|+	\sup\limits_{a\in\mathbb{D}} (1-|a|^{2})^{\frac{1-\lambda}{2}}\|f\circ\sigma_a-f(a)\|_{H^2}<\infty,$$
	where $$\sigma_{a}(z)=\frac{a-z}{1-\overline{a}z}.$$} 
\end{definition}
It is well known that $\mathcal{L}^{2,\lambda }$ is a Banach space under the norm $\left \| f \right \|_{\mathcal{L}^{2,\lambda } }.$

\begin{lemma}\label{equivalent conditions} {\rm (\cite{LL,WJ,XJ2})}
	For any function $f \in \mathcal{L}^{2, \lambda}(\mathbb{D}) $, the following norm equivalences hold: 
	\begin{align*}
		(i)\, & \|f\|_{\mathcal{L}^{2, \lambda}}\approx|f(0)|+\sup _{I \subset \partial \mathbb{D}}\left(\frac{1}{|I|^{\lambda}} \int_{S(I)}\left|f^{\prime}(z)\right|^{2}\left(1-|z|^{2}\right) d A(z)\right)^{1 / 2}  ;\\
		(ii)\, &\|f\|_{\mathcal{L}^{2, \lambda}}\approx|f(0)|+\sup _{a \in \mathbb{D}}\left(\left(1-|a|^{2}\right)^{1-\lambda} \int_{\mathbb{D}}\left|f^{\prime}(z)\right|^{2}\left(1-\left|\sigma_{a}(z)\right|^{2}\right) d A(z)\right)^{1 / 2} ;\\
		(iii)\, &\|f\|_{\mathcal{L}^{2, \lambda}}\approx|f(0)|+\sup _{a \in \mathbb{D}}\left(\left(1-|a|^{2}\right)^{1-\lambda} \int_{\mathbb{D}}\left|f^{\prime}(z)\right|^{2} \log \frac{1}{\left|\sigma_{a}(z)\right|} d A(z)\right)^{1 / 2}.
	\end{align*}
\end{lemma}

\begin{definition} {\rm (\cite{YL})
Given $0<\lambda< 2$ and $p>0$, a function $f$ in $Hol(\mathbb{D})$ is said to belong to Bergman-Morrey space $A^{p,\lambda}$ if
	$$\|f\|_{A^{p,\lambda}}:=\left(|f(0)|^p+	\sup\limits_{a\in\mathbb{D}} (1-|a|^{2})^{2-\lambda}\int_{\mathbb{D}}|f'(z)|^{p}(1-|z|^{2})^{p}|\sigma'_{a}(z)|^{2}dA(z)\right)^{\frac{1}{p}}<\infty.$$}
\end{definition}

\begin{lemma} {\rm (\cite[Remark 2.2]{YL})}
Let  $0<\lambda<2$  and  $p>0 ,$ then  $f \in Hol(\mathbb{D})$  has the following equivalent norms in  $A^{p, \lambda} .$
\begin{align*}
\|f\|_{A^{p, \lambda}} & \approx|f(0)|+\sup _{a \in \mathbb{D}}\left(\left(1-|a|^{2}\right)^{2-\lambda} \int_{\mathbb{D}}\left|f^{\prime}(z)\right|^{p}\left(1-|z|^{2}\right)^{p}\left|\sigma_{a}^{\prime}(z)\right|^{2} d A(z)\right)^{\frac{1}{p}} \\
& \approx|f(0)|+\sup _{I \subset \partial \mathbb{D}}\left(\frac{1}{|I|^{\lambda}} \int_{S(I)}\left|f^{\prime}(z)\right|^{p}\left(1-|z|^{2}\right)^{p} d A(z)\right)^{\frac{1}{p}} \\
& \approx|f(0)|+\sup _{a \in \mathbb{D}}\left(\int_{\mathbb{D}}\left|f^{\prime}(z)\right|^{p}\left(1-|z|^{2}\right)^{p-\lambda}\left(1-\left|\sigma_{a}(z)\right|^{2}\right)^{\lambda} d A(z)\right)^{\frac{1}{p}}.
\end{align*}
\end{lemma}

Throughout this paper, we write $z=b e^{i\theta}$ with $b=|z|\in[0,1)$.
For a (sufficiently smooth) complex-valued harmonic function $f$ on $\mathbb D$,
we denote by $f_b=\partial f/\partial b$ and $f_\theta=\partial f/\partial \theta$
the radial and angular derivatives, respectively, and we use the product $b f_b=b\,\partial f/\partial b$.
This notation is chosen to avoid confusion with the exponent parameter $r$ appearing in spaces such as $M(r,q,s)$ and $F(r,q,s)$.
Moreover, in terms of Wirtinger derivatives, one has the identities
\[
-i f_{\theta}=z f_{z}-\overline{z}\, f_{\bar z},\quad
b f_{b}=z f_{z}+\overline{z}\, f_{\bar z}.
\]
In particular, if $f=h+\overline g$ with $h,g\in H(\mathbb D)$, then $f_z=h'$ and $f_{\bar z}=\overline{g'}$, and hence
\[
-i f_{\theta}=z h'(z)-\overline{z}\,\overline{g'(z)},\quad
b f_{b}=z h'(z)+\overline{z}\,\overline{g'(z)}.
\]

\section{Growth of conjugate-type functions for harmonic quasiregular mappings}\label{section2}

In this section, we establish conjugate-type stability for harmonic quasiregular mappings in several families of function spaces.
Our first goal is to show that, under quasiregularity, membership of the real part in a given scale forces the same membership for the imaginary part, with quantitative control.

\begin{theorem}\label{v<u}
Let $\alpha>-1$ and $\alpha+1<p<\alpha+2$.
Suppose that $f=u+iv$ is a harmonic $K$-quasiregular mapping in $\mathbb D$.
Assume that $u\in Q_h(1,p,\alpha)$,
then $v\in Q_h(1,p,\alpha)$ and
$$\left \|v  \right \| _{Q_h(1,p,\alpha)}\le K\left \|u  \right \| _{Q_h(1,p,\alpha)}.$$
\end{theorem}

\begin{proof}
Without loss of generality, assume that $f(0)=0$. Write $f=h+\overline{g}$, and for each $a\in\mathbb D$, define
$$f_{\sigma_a}:=f\circ\sigma_a=h_{\sigma_a}+\overline{g_{\sigma_a}},$$ where $h_{\sigma_a}=h\circ\sigma_a$ and $g_{\sigma_a}=g\circ\sigma_a$.

Since $u\in Q_h(1,p,\alpha)$, it follows that for every $a\in\mathbb D$,
$$
\int_{\mathbb{D}}\bigl|\nabla(u\circ\sigma_a)(z)\bigr|^{p}(1-|z|^{2})^{\alpha}\,dA(z)<\infty.
$$

Set $F:=h+g$ and $G:=h-g$. Then
$$u=\Re(f)=\Re(h+\overline{g})=\Re(h+g)=\Re(F)$$
and
$$v=\Im(f)=\Im(h+\overline{g})=\Im(h-g)=\Im(G).$$
Consequently,
$u\circ\sigma_a=\Re(F\circ\sigma_a)$ and $v\circ\sigma_a=\Im(G\circ\sigma_a)$, and hence
$$
\bigl|(F\circ\sigma_a)'(z)\bigr|=\bigl|\nabla(u\circ\sigma_a)(z)\bigr|,
\quad
\bigl|(G\circ\sigma_a)'(z)\bigr|=\bigl|\nabla(v\circ\sigma_a)(z)\bigr|.
$$

Let $w={g'}/{h'}$. Since $f$ is $K$-quasiregular, we have $|w|\le k<1$, where $k$ is given by \eqref{kandK}.
Moreover,
$$
F'=h'+g'=h'(1+w),
\quad
G'=h'-g'=h'(1-w).
$$
Using the elementary inequalities $$|1-w|\le 1+|w|$$ and $$|1+w|\ge 1-|w|\ge 1-k,$$ we obtain
$$
|G'(z)|\le |h'(z)|(1+k),
\quad
|F'(z)|\ge |h'(z)|(1-k),
$$
and therefore
$$
|G'(z)|\le \frac{1+k}{1-k}\,|F'(z)|=K|F'(z)|
\quad (z\in\mathbb D).
$$

For any $a\in\mathbb D$, by the chain rule, we have
$$
\bigl|(G\circ\sigma_a)'(z)\bigr|
=\bigl|G'(\sigma_a(z))\bigr|\,\bigl|\sigma_a'(z)\bigr|
\le K\bigl|F'(\sigma_a(z))\bigr|\,\bigl|\sigma_a'(z)\bigr|
=K\bigl|(F\circ\sigma_a)'(z)\bigr|.
$$
Raising to the power $p$, multiplying by $(1-|z|^{2})^{\alpha}$, integrating over $\mathbb D$, and taking the supremum over $a\in\mathbb D$, we obtain
$$
\sup_{a\in\mathbb{D}}\int_{\mathbb{D}}\bigl|(G\circ\sigma_a)'(z)\bigr|^{p}(1-|z|^{2})^{\alpha}\,dA(z)
\le K^{p}\sup_{a\in\mathbb{D}}\int_{\mathbb{D}}\bigl|(F\circ\sigma_a)'(z)\bigr|^{p}(1-|z|^{2})^{\alpha}\,dA(z).
$$
Using $$\bigl|(F\circ\sigma_a)'(z)\bigr|=\bigl|\nabla(u\circ\sigma_a)(z)\bigr|$$ and $$\bigl|(G\circ\sigma_a)'(z)\bigr|=\bigl|\nabla(v\circ\sigma_a)(z)\bigr|,$$ this shows that $v\in Q_h(1,p,\alpha)$ and yields the desired estimate.
\end{proof}

The next theorem shows that the same principle persists in the three-parameter family $F_h(p,q,s)$, which unifies several classical spaces as its special cases.

\begin{theorem}\label{v<u-F}
Let $0< p<\infty$, $-2<q<\infty$ and $0<s<\infty$ satisfy $q+s>-1$.
Suppose that $f=u+iv$ is a harmonic $K$-quasiregular mapping in $\mathbb D$.
Assume that $u\in F_h(p,q,s)$,
then $v\in F_h(p,q,s)$ and
$$\left \|v  \right \| _{F_h(p,q,s)}\le K\left \|u  \right \| _{F_h(p,q,s)}.$$
\end{theorem}

\begin{proof}
Without loss of generality, assume that $f(0)=0$. Write $f=h+\overline{g}$ and set
$$F:=h+g,\quad G:=h-g.$$

Since $u\in F_h(p,q,s)$, it follows that for every $a\in\mathbb D$,
$$
\int_{\mathbb{D}}|\nabla u(z)|^{p}(1-|z|^{2})^{q}\bigl(1-|\sigma_a(z)|^{2}\bigr)^{s}\,dA(z)<\infty.
$$
Moreover,
$$
|F'(z)|=\bigl|\nabla u(z)\bigr|,
\quad
|G'(z)|=\bigl|\nabla v(z)\bigr|.
$$
By the same argument as in the proof of Theorem~\ref{v<u}, we have
$$
|G'(z)|\le K|F'(z)|\quad (z\in\mathbb D).
$$
Therefore, for every $a\in\mathbb D$,
$$
\int_{\mathbb{D}}|G'(z)|^{p}(1-|z|^{2})^{q}\bigl(1-|\sigma_a(z)|^{2}\bigr)^{s}\,dA(z)
\le K^{p}\int_{\mathbb{D}}|F'(z)|^{p}(1-|z|^{2})^{q}\bigl(1-|\sigma_a(z)|^{2}\bigr)^{s}\,dA(z).
$$
Taking the supremum over $a\in\mathbb D$, we obtain
$$
\sup_{a\in\mathbb{D}}\int_{\mathbb{D}}|G'(z)|^{p}(1-|z|^{2})^{q}\bigl(1-|\sigma_a(z)|^{2}\bigr)^{s}\,dA(z)
\le K^{p}\sup_{a\in\mathbb{D}}\int_{\mathbb{D}}|F'(z)|^{p}(1-|z|^{2})^{q}\bigl(1-|\sigma_a(z)|^{2}\bigr)^{s}\,dA(z).
$$
Using $|F'(z)|=\bigl|\nabla u(z)\bigr|$ and $|G'(z)|=\bigl|\nabla v(z)\bigr|$, this shows that $v\in F_h(p,q,s)$ and yields the desired assertion.
\end{proof}

For each analytic space $X$ considered below whose (quasi-)norm can be written in terms of an area integral
involving $|f|$ (or equivalently $|f|^p$), we denote by $X_h$ its \emph{harmonic counterpart} obtained by keeping
the same defining formula but dropping analyticity and requiring only that $f$ be harmonic in $\mathbb D$.
In particular, the harmonic Morrey space $\mathcal L_h^{2,\lambda}$, the harmonic Bergman--Morrey space
$A_h^{p,\lambda}$, and the harmonic $Q_s$-space $Q_{s,h}$ are understood in this sense (cf.\ the definition of
$F_h(p,q,s)$), and we do not repeat their definitions.

By specializing the parameters $(p,q,s)$ to the standard identifications of Morrey-type spaces, Bergman-Morrey-type spaces, and $Q_s$-type spaces, we obtain the following immediate consequences.

\begin{coro}\label{coro:morrey-K}
Let $0<\lambda<1$. Suppose that $f=u+iv$ is a harmonic $K$-quasiregular mapping in $\mathbb D$.
Assume that $u\in \mathcal{L}_{h}^{2,\lambda}$,
then $v\in \mathcal{L}_{h}^{2,\lambda}$ and
$\left \|v  \right \| _{\mathcal{L}_{h}^{2,\lambda}}\le K\left \|u  \right \| _{\mathcal{L}_{h}^{2,\lambda}}.$
\smallskip
\end{coro}

\begin{remark} {\rm This is the specialization of Theorem~\ref{v<u-F} to the case
$\mathcal{L}^{2,\lambda}=F(2,1-\lambda,\lambda)$ (with equivalent norms), see e.g., \cite{LL,WJ,ZR,ZR1}.}
\end{remark}

\begin{coro}\label{coro:bergman-morrey-K}
Let $0<\lambda<2$ and $0<p<\infty$. Suppose that $f=u+iv$ is a harmonic $K$-quasiregular mapping in $\mathbb D$.
Assume that $u\in A_{h}^{p,\lambda}$,
then $v\in A_{h}^{p,\lambda}$ and
$\left \|v  \right \| _{A_{h}^{p,\lambda}}\le K\left \|u  \right \| _{A_{h}^{p,\lambda}}.$
\smallskip
\end{coro}

\begin{remark} {\rm This is the specialization of Theorem~\ref{v<u-F} to the case
$A^{p,\lambda}=F(p,p-\lambda,\lambda)$ (with equivalent norms), see e.g., \cite{YL,ZR,ZR1}.}
\end{remark}

\begin{coro}\label{coro:Qs-K}
Let $s>0$. Suppose that $f=u+iv$ is a harmonic $K$-quasiregular mapping in $\mathbb D$.
Assume that $u\in Q_{s,h}$,
then $v\in Q_{s,h}$ and
$\left \|v  \right \| _{Q_{s,h}}\le K\left \|u  \right \| _{Q_{s,h}}.$
\smallskip
\end{coro}

\begin{remark}  {\rm This is the specialization of Theorem~\ref{v<u-F} to the case
$Q_s=F(2,0,s)$ (with equivalent norms), see e.g., \cite{XJ2,ZR,ZR1}.}
\end{remark}

{\rm The preceding results correspond to the case $K'=0$.
We now extend them to the harmonic $(K,K')$-quasiregular setting, where an additional inhomogeneous term leads to an explicit additive contribution in the norm estimates.}

\begin{theorem}\label{thm:Qh-CK21}
Let $\alpha>-1$ and $\alpha+1<p<\alpha+2$. Suppose that $f=u+iv$ is a harmonic $(K,K')$-quasiregular mapping in $\mathbb D$. Assume that $u\in Q_h(1,p,\alpha)$. Then $v\in Q_h(1,p,\alpha)$ and
\[
\|v\|_{Q_h(1,p,\alpha)}^{p}\le 2^{\max\{p-1,0\}}
\left(K^{p}\,\|u\|_{Q_h(1,p,\alpha)}^{p}+(K')^{p/2}\,C_{p,\alpha}\right),
\]
where
\[
C_{p,\alpha}:=\sup_{a\in\mathbb D}\int_{\mathbb D}|\sigma_a'(z)|^{p}(1-|z|^{2})^{\alpha}\,dA(z)<\infty.
\]
\end{theorem}

\begin{proof}
Without loss of generality, assume that $f(0)=0$. Write $f=h+\overline g$ with $h,g\in H(\mathbb D)$ and $g(0)=0$. Set
\[
F_1:=h+g,\quad F_2:=h-g.
\]
Then $u=\Re(f)=\Re(F_1)$ and $v=\Im(f)=\Im(F_2)$. For each $a\in\mathbb D$, we have
\[
|(F_1\circ\sigma_a)'(z)|=|\nabla(u\circ\sigma_a)(z)|,\quad |(F_2\circ\sigma_a)'(z)|=|\nabla(v\circ\sigma_a)(z)|.
\]

Since $f$ is $(K,K')$-quasiregular, we have
\[
\Lambda_f(z)\le K\lambda_f(z)+\sqrt{K'}\quad (z\in\mathbb D).
\]
Noting that $f_z=h'$ and $f_{\bar z}=\overline{g'}$, we obtain
\[
|g'(z)|\le \mu_1 |h'(z)|+\mu_2,\quad 
\mu_1=\frac{K-1}{K+1},\quad \mu_2=\frac{\sqrt{K'}}{1+K}.
\]
Consequently,
\[
|F_2'(z)|\le \Lambda_f(z)\le (1+\mu_1)|h'(z)|+\mu_2,
\quad
|F_1'(z)|\ge \lambda_f(z)\ge (1-\mu_1)|h'(z)|-\mu_2,
\]
and hence
\[
|F_2'(z)|\le \frac{1+\mu_1}{1-\mu_1}|F_1'(z)|+\frac{2\mu_2}{1-\mu_1}
=K\,|F_1'(z)|+\sqrt{K'}\quad (z\in\mathbb D).
\]
Therefore, for every $a\in\mathbb D$ and $z\in\mathbb D$, we have
\begin{align*}
|(F_2\circ\sigma_a)'(z)|
&=|F_2'(\sigma_a(z))|\,|\sigma_a'(z)|
\le \bigl(K|F_1'(\sigma_a(z))|+\sqrt{K'}\bigr)\,|\sigma_a'(z)|  \\
&=K|(F_1\circ\sigma_a)'(z)|+\sqrt{K'}\,|\sigma_a'(z)|.
\end{align*}
Raising to the power $p$ and using $$(A+B)^p\le 2^{\max\{p-1,0\}}(A^p+B^p),$$ we obtain
\[
|(F_2\circ\sigma_a)'(z)|^{p}
\le 2^{\max\{p-1,0\}}
\left(
K^{p}|(F_1\circ\sigma_a)'(z)|^{p}
+(K')^{p/2}|\sigma_a'(z)|^{p}
\right).
\]
Multiplying by $(1-|z|^{2})^{\alpha}$, integrating over $\mathbb D$, and taking the supremum over $a\in\mathbb D$, we get
\begin{align*}
\|v\|_{Q_h(1,p,\alpha)}^{p}
&=\sup_{a\in\mathbb D}\int_{\mathbb D}|(F_2\circ\sigma_a)'(z)|^{p}(1-|z|^{2})^{\alpha}\,dA(z)\\
&\le 2^{\max\{p-1,0\}}
\left(
K^{p}\sup_{a\in\mathbb D}\int_{\mathbb D}|(F_1\circ\sigma_a)'(z)|^{p}(1-|z|^{2})^{\alpha}\,dA(z)
+(K')^{p/2}\,C_{p,\alpha}
\right)\\
&=2^{\max\{p-1,0\}}
\left(
K^{p}\,\|u\|_{Q_h(1,p,\alpha)}^{p}+(K')^{p/2}\,C_{p,\alpha}
\right),
\end{align*}
which proves the desired assertion.
\end{proof}

\begin{theorem}\label{thm:Fh-KK'}
Let $0<p<\infty$, $-2<q<\infty$ and $0<s<\infty$ satisfy $q+s>-1$.
Suppose that $f=u+iv$ is a harmonic $(K,K')$-quasiregular mapping in $\mathbb D$.
Assume that $u\in F_h(p,q,s)$. Then $v\in F_h(p,q,s)$ and
\[
\|v\|_{F_h(p,q,s)}^{p}\le 2^{\max\{p-1,0\}}
\left(K^{p}\,\|u\|_{F_h(p,q,s)}^{p}+(K')^{p/2}\,C_{q,s}\right),
\]
where
\[
C_{q,s}:=\sup_{a\in\mathbb D}\int_{\mathbb D}(1-|z|^{2})^{q}\bigl(1-|\sigma_a(z)|^{2}\bigr)^{s}\,dA(z)<\infty.
\]
\end{theorem}

\begin{proof}
Using the same notation and method as in the proof of Theorem \ref{thm:Qh-CK21},
we obtain
\[
|F_2'(z)|^{p}
\le 2^{\max\{p-1,0\}}
\left(
K^{p}|F_1'(z)|^{p}
+(K')^{p/2}
\right).
\]
Multiplying by $$(1-|z|^{2})^{q}\bigl(1-|\sigma_a(z)|^{2}\bigr)^{s},$$ integrating over $\mathbb D$, and taking the supremum over $a\in\mathbb D$, we get
\begin{align*}
\|v\|_{F_h(p,q,s)}^{p}
&=\sup_{a\in\mathbb D}\int_{\mathbb D}|F_2'(z)|^{p}(1-|z|^{2})^{q}\bigl(1-|\sigma_a(z)|^{2}\bigr)^{s}\,dA(z)\\
&\le 2^{\max\{p-1,0\}}
\left(
K^{p}\sup_{a\in\mathbb D}\int_{\mathbb D}|F_1'(z)|^{p}(1-|z|^{2})^{q}\bigl(1-|\sigma_a(z)|^{2}\bigr)^{s}\,dA(z)
+(K')^{p/2}\,C_{q,s}
\right)\\
&=2^{\max\{p-1,0\}}
\left(
K^{p}\,\|u\|_{F_h(p,q,s)}^{p}+(K')^{p/2}\,C_{q,s}
\right),
\end{align*}
which proves the desired assertion of Theorem \ref{thm:Fh-KK'}.
\end{proof}

\begin{coro}\label{coro:morrey-KK'}
Let $0<\lambda<1$. Suppose that $f=u+iv$ is a harmonic $(K,K')$-quasiregular mapping in $\mathbb D$.
Assume that $u\in \mathcal{L}_{h}^{2,\lambda}$.
Then $v\in \mathcal{L}_{h}^{2,\lambda}$ and
\[
\|v\|_{\mathcal{L}_{h}^{2,\lambda}}^{2}
\le
2\left(
K^{2}\,\|u\|_{\mathcal{L}_{h}^{2,\lambda}}^{2}+K'\,C_{\lambda}
\right),
\]
where
\[
C_{\lambda}:=\sup_{a\in\mathbb D}\int_{\mathbb D}
(1-|z|^{2})^{1-\lambda}\bigl(1-|\sigma_a(z)|^{2}\bigr)^{\lambda}\,dA(z)<\infty.
\]
\end{coro}

\begin{coro}\label{coro:bergman-morrey-KK'}
Let $0<\lambda<2$ and $0<p<\infty$. Suppose that $f=u+iv$ is a harmonic $(K,K')$-quasiregular mapping in $\mathbb D$.
Assume that $u\in A_{h}^{p,\lambda}$.
Then $v\in A_{h}^{p,\lambda}$ and
\[
\|v\|_{A_{h}^{p,\lambda}}^{p}
\le
2^{\max\{p-1,0\}}
\left(
K^{p}\,\|u\|_{A_{h}^{p,\lambda}}^{p}+(K')^{p/2}\,C_{p,\lambda}
\right),
\]
where
\[
C_{p,\lambda}:=\sup_{a\in\mathbb D}\int_{\mathbb D}
(1-|z|^{2})^{p-\lambda}\bigl(1-|\sigma_a(z)|^{2}\bigr)^{\lambda}\,dA(z)<\infty.
\]
\end{coro}

\begin{remark}
{\rm This is the specialization of Theorem~\ref{thm:Fh-KK'} to the case
$A^{p,\lambda}=F(p,p-\lambda,\lambda)$ (with equivalent norms), see e.g., \cite{YL,ZR,ZR1}.}
\end{remark}

\begin{coro}\label{coro:Qs-KK'}
Let $s>0$. Suppose that $f=u+iv$ is a harmonic $(K,K')$-quasiregular mapping in $\mathbb D$.
Assume that $u\in Q_{s,h}$.
Then $v\in Q_{s,h}$ and
\[
\|v\|_{Q_{s,h}}^{2}
\le
2\left(
K^{2}\,\|u\|_{Q_{s,h}}^{2}+K'\,C_{s}
\right),
\]
where
\[
C_{s}:=\sup_{a\in\mathbb D}\int_{\mathbb D}
\bigl(1-|\sigma_a(z)|^{2}\bigr)^{s}\,dA(z)<\infty.
\]
\end{coro}

\section{Harmonic quasiconformal mappings in function spaces}\label{section3}

\rm{In this section, we derive integrability and growth consequences for normalized harmonic quasiconformal mappings $f$ in $\mathcal S_H(K)$.
The key input is the distortion theory encoded in the order $\alpha_K$, which governs the admissible parameter ranges in the $M_h(p,q,s)$ and $F_h(p
,q,s)$ scales considered below.}

\begin{theorem}\label{M-th3}
Suppose that $f\in \mathcal{S}_H(K)$ for some $K\geq1$. Then
\begin{enumerate}
\item[(i)] $f\in M_h(p,q,s)$ for all
\begin{equation}\label{M_h(p,q,s)-p}
p<\min\left\{\frac{q+s+1}{\alpha_K},\,\frac{q+2}{\alpha_K}\right\};
\end{equation}
\item[(ii)] $f_z,\overline{f_{\bar z}}\in M(r,q,s)$, $f_{\theta},\,b f_b\in M_h(r,q,s)$ for all
\begin{equation}\label{M_h(p,q,s)-r}
r<\min\left\{\frac{q+s+1}{\alpha_K+1},\,\frac{q+2}{\alpha_K+1}\right\}.
\end{equation}
\end{enumerate}
\end{theorem}

\begin{proof}
We write $f=h+\overline{g}$, where $h$ and $g$ are analytic in $\mathbb{D}$.

\medskip
{(i)} Fix $\zeta\in\mathbb{D}$. Define
\[
F_\zeta(z)=\frac{f\!\left(\frac{z+\zeta}{1+\overline{\zeta}z}\right)-f(\zeta)}{(1-|\zeta|^2)\,h'(\zeta)}\quad (z\in\mathbb{D}).
\]
Then $F_\zeta\in \mathcal{S}_H(K)$, and
\[
F_{zz}(0)=(1-|\zeta|^2)\frac{h''(\zeta)}{h'(\zeta)}-2\overline{\zeta}.
\]
Since $|F_{zz}(0)|\le 2\alpha_K$, we obtain
\[
\left|\frac{h''(\zeta)}{h'(\zeta)}\right|\le \frac{C}{1-|\zeta|}\quad (\zeta\in\mathbb{D}),
\]
equivalently,
\[
\left|\frac{h''(re^{i\theta})}{h'(re^{i\theta})}\right|\le \frac{C}{1-r}\quad (0\le r<1).
\]
On the other hand, arguing as in \cite[p.\,98]{DP},
\[
|h'(re^{i\theta})|\le \frac{(1+r)^{\alpha_K-1}}{(1-r)^{\alpha_K+1}},
\quad
|h'(z)|\lesssim \frac{1}{(1-|z|^2)^{\alpha_K+1}}.
\]
Since $w={g'}/{h'}$ satisfies $|w|\le k<1$, we have
\[
|g'(z)|\le k|h'(z)|\lesssim \frac{1}{(1-|z|^2)^{\alpha_K+1}}.
\]
Moreover, integrating $h'$ radially yields
\[
|h(re^{i\theta})-h(0)|
\le \int_0^r |h'(te^{i\theta})|\,dt
\lesssim \int_0^r \frac{1}{(1-t)^{\alpha_K+1}}\,dt
\lesssim (1-r)^{-\alpha_K},
\]
which implies that
\[
|h(z)|\lesssim \frac{1}{(1-|z|^2)^{\alpha_K}},
\quad
|g(z)|\lesssim \frac{1}{(1-|z|^2)^{\alpha_K}},
\]
therefore,
\[
|f(z)|\le |h(z)|+|g(z)|\lesssim \frac{1}{(1-|z|^2)^{\alpha_K}}.
\]

For $a\in\mathbb D$, set
\begin{align*}
I_p(a)
&:=\int_{\mathbb{D}} |f(z)|^{p}(1-|z|^2)^{q}\,(1-|\sigma_a(z)|^2)^{s}\,dA(z).
\end{align*}
By \cite[Lemma 3.10]{ZK}, we see that
\begin{align*}
I_p(a)
&\lesssim (1-|a|^2)^{s}\int_{\mathbb{D}}
\frac{(1-|z|^2)^{q+s-p\alpha_K}}{|1-\overline{a}z|^{2s}}\,dA(z).
\end{align*}
By setting
\[
t:=q+s-p\alpha_K,\quad c:=s-q-2+p\alpha_K,
\]
then $$2+t+c=2s.$$ If $t>-1$, \cite[Lemma 3.10]{ZK} yields the three cases
\[
\int_{\mathbb{D}}
\frac{(1-|z|^2)^{t}}{|1-\overline{a}z|^{2+t+c}}\,dA(z)
\lesssim
\begin{cases}
1, & c<0,\\[1mm]
\log\!\dfrac{1}{1-|a|^2}, & c=0,\\[2mm]
(1-|a|^2)^{-c}, & c>0,
\end{cases}
\]
and consequently,
\[
I_p(a)\lesssim
\begin{cases}
(1-|a|^2)^s, & c<0,\\[1mm]
(1-|a|^2)^s\log\!\dfrac{1}{1-|a|^2}, & c=0,\\[2mm]
(1-|a|^2)^{s-c}=(1-|a|^2)^{2+q-p\alpha_K}, & c>0.
\end{cases}
\]
Suppose that \eqref{M_h(p,q,s)-p} holds. Then $t>-1$ and $2+q-p\alpha_K>0$, so each of the above three bounds is uniformly bounded in $a\in\mathbb D$ (note that $s>0$).
Therefore, $$\sup_{a\in\mathbb D} I_p(a)<\infty,$$ which shows that $f\in M_h(p,q,s)$.

\medskip
{(ii)} From part (i), we find that 
\[
|h'(z)|\lesssim \frac{1}{(1-|z|^{2})^{\alpha_K+1}},
\]
hence
\[
|h'(z)|^{r}\lesssim \frac{1}{(1-|z|^{2})^{r(\alpha_K+1)}}.
\]
For $a\in\mathbb D$, set
\begin{align*}
I_r(a)
&:=\int_{\mathbb{D}} |h'(z)|^{r}(1-|z|^{2})^{q}\,(1-|\sigma_a(z)|^{2})^{s}\,dA(z)\\
&\lesssim (1-|a|^{2})^{s}\int_{\mathbb{D}}
\frac{(1-|z|^{2})^{q+s-r(\alpha_K+1)}}{|1-\overline{a}z|^{2s}}\,dA(z).
\end{align*}
If we put
\[
t:=q+s-r(\alpha_K+1),\quad c:=s-q-2+r(\alpha_K+1),
\]
then $$2+t+c=2s.$$ Applying \cite[Lemma 3.10]{ZK} exactly as in part (i), we obtain (for $t>-1$)
\[
I_r(a)\lesssim
\begin{cases}
(1-|a|^2)^s, & c<0,\\[1mm]
(1-|a|^2)^s\log\!\dfrac{1}{1-|a|^2}, & c=0,\\[2mm]
(1-|a|^2)^{s-c}=(1-|a|^2)^{2+q-r(\alpha_K+1)}, & c>0.
\end{cases}
\]
Assume that \eqref{M_h(p,q,s)-r} holds. Then $t>-1$ and $2+q-r(\alpha_K+1)>0$. Hence $$\sup\limits_{a\in\mathbb D}I_r(a)<\infty$$ and $f_z=h'\in M(r,q,s)$.

Next, since $|g'(z)|\le k|h'(z)|$, we know that
\[
|g'(z)|^{r}\le k^{r}|h'(z)|^{r},
\quad
\|g'\|_{M(r,q,s)}\lesssim k\|h'\|_{M(r,q,s)},
\]
which implies that $\overline{f_{\bar z}}=g'\in M(r,q,s)$.

Finally, using
\[
-i f_{\theta}=zh'(z)-\overline{z}\,\overline{g'(z)},\quad
bf_{b}=zh'(z)+\overline{z}\,\overline{g'(z)},
\]
we have
\[
|f_{\theta}(z)|\le |h'(z)|+|g'(z)|\le (1+k)\,|h'(z)|,
\quad
|b f_{b}(z)|\le |h'(z)|+|g'(z)|\le (1+k)\,|h'(z)|.
\]
Raising to the power $r$ and using $h'\in M(r,q,s)$ yields $f_{\theta},\, b f_{b}\in M_{h}(r,q,s)$.
This completes the proof of Theorem \ref{M-th3}.
\end{proof}

\begin{theorem}\label{F-th3}
Suppose that $f\in \mathcal{S}_H(K)$ for some $K\geq1$. Then
\begin{enumerate}
\item[(i)] $f\in F_h(p,q,s)$ for all
\begin{equation}\label{F_h(p,q,s)-p}
p<\min\left\{\frac{q+s+1}{\alpha_K +1},\,\frac{q+2}{\alpha_K +1}\right\};
\end{equation}
\item[(ii)] $f_z,\overline{f_{\bar z}}\in F(r,q,s)$, $f_{\theta},\,b f_b\in F_h(r,q,s)$ for all
\begin{equation}\label{F_h(p,q,s)-r}
r<\min\left\{\frac{q+s+1}{\alpha_K+2},\,\frac{q+2}{\alpha_K+2}\right\}.
\end{equation}
\end{enumerate}
\end{theorem}

\begin{proof}
We keep the notation of the statement and write $f=h+\overline g$ with $h,g\in H(\mathbb D)$.

By the proof of Theorem \ref{M-th3}, we obtain
\[
|h'(z)|\lesssim \frac{1}{(1-|z|^2)^{\alpha_K+1}},
\quad
\left|\frac{h''(z)}{h'(z)}\right|\lesssim \frac{1}{1-|z|^2}.
\]
Let $w={g'}/{h'}$ be the analytic dilatation of $f$. Since $|w|\le k<1$ in $\mathbb D$, we get
\[
|g'(z)|\le k|h'(z)|\lesssim \frac{1}{(1-|z|^2)^{\alpha_K+1}}\quad (z\in\mathbb D).
\]
Moreover,
\[
|h''(z)|\le \left|\frac{h''(z)}{h'(z)}\right|\,|h'(z)|
\lesssim \frac{1}{(1-|z|^2)^{\alpha_K+2}}.
\]
Applying Schwarz--Pick lemma to $\varphi={w}/{k}$ gives $$|w'(z)|\lesssim \frac{1}{1-|z|^2},$$ and hence
\[
|g''(z)|
=|w'(z)h'(z)+w(z)h''(z)|
\lesssim \frac{1}{(1-|z|^2)^{\alpha_K+2}}.
\]
Since $f_z=h'$ and $f_{\bar z}=\overline{g'}$, we have
\[
|f_z(z)|+|f_{\bar z}(z)|
=|h'(z)|+|g'(z)|
\le (1+k)|h'(z)|
\lesssim \frac{1}{(1-|z|^2)^{\alpha_K+1}}.
\]

\medskip
{(i)} For $a\in\mathbb D$, set
\[
I_p(a):=\int_{\mathbb D}\bigl(|f_z(z)|+|f_{\bar z}(z)|\bigr)^p(1-|z|^2)^q(1-|\sigma_a(z)|^2)^s\,dA(z).
\]
Then
\begin{align*}
I_p(a)
&\lesssim \int_{\mathbb D}(1-|z|^2)^{q-p(\alpha_K+1)}(1-|\sigma_a(z)|^2)^s\,dA(z)\\
&= (1-|a|^2)^s\int_{\mathbb D}\frac{(1-|z|^2)^{q+s-p(\alpha_K+1)}}{|1-\overline a z|^{2s}}\,dA(z).
\end{align*}
If we take
\[
t:=q+s-p(\alpha_K+1),\quad c:=s-q-2+p(\alpha_K+1),
\]
then $2+t+c=2s$. Applying \cite[Lemma 3.10]{ZK} as in the proof of (i) of Theorem \ref{M-th3}, we obtain (for $t>-1$)
\[
I_p(a)\lesssim
\begin{cases}
(1-|a|^2)^s, & c<0,\\[1mm]
(1-|a|^2)^s\log\!\dfrac{1}{1-|a|^2}, & c=0,\\[2mm]
(1-|a|^2)^{s-c}=(1-|a|^2)^{q+2-p(\alpha_K+1)}, & c>0.
\end{cases}
\]
Assume that \eqref{F_h(p,q,s)-p} holds. Then $t>-1$ and $q+2-p(\alpha_K+1)>0$. Hence $$\sup\limits_{a\in\mathbb D}I_p(a)<\infty,$$
which shows that $f\in F_h(p,q,s)$.

\medskip
{(ii)} Since $f_z=h'$ is analytic, it suffices to show that
\[
\sup_{a\in\mathbb D}\int_{\mathbb D}|h''(z)|^r(1-|z|^2)^q(1-|\sigma_a(z)|^2)^s\,dA(z)<\infty.
\]
Using $$|h''(z)|\lesssim \frac{1}{(1-|z|^2)^{\alpha_K+2}}$$ and arguing as above, we reduce to the same three-case estimate in \cite[Lemma 3.10]{ZK} with
\[
t:=q+s-r(\alpha_K+2),\quad c:=s-q-2+r(\alpha_K+2),
\]
and obtain uniform boundedness in $a$ provided
\[
q+s-r(\alpha_K+2)>-1,\quad q+2-r(\alpha_K+2)>0,
\]
which is equivalent to \eqref{F_h(p,q,s)-r}. Hence $f_z=h'\in F(r,q,s)$.

Since $\overline{f_{\bar z}}=g'$, it remains to verify $g'\in F(r,q,s)$, i.e., the defining integral condition for $g''$.
This follows from $$|g''(z)|\lesssim \frac{1}{(1-|z|^2)^{\alpha_K+2}}$$ by the same argument and under the same condition \eqref{F_h(p,q,s)-r}.

Finally, using
\[
-i f_{\theta}=z h'(z)-\overline{z}\,\overline{g'(z)},\quad
bf_{b}=z h'(z)+\overline{z}\,\overline{g'(z)},
\]
differentiation yields
\[
|(f_\theta)_z|+|(f_\theta)_{\bar z}|
\lesssim |h'(z)|+|h''(z)|+|g'(z)|+|g''(z)|,
\]
and similarly for $|(bf_b)_z|+|(bf_b)_{\bar z}|$. Hence by the above pointwise bounds, we obtain
\[
\bigl(|(f_\theta)_z|+|(f_\theta)_{\bar z}|\bigr)^r
\lesssim (1-|z|^2)^{-r(\alpha_K+2)},
\quad
\bigl(|(bf_b)_z|+|(bf_b)_{\bar z}|\bigr)^r
\lesssim (1-|z|^2)^{-r(\alpha_K+2)}.
\]
Applying \cite[Lemma 3.10]{ZK} once more as above, we conclude that
\[
f_\theta,\,bf_b\in F_h(r,q,s)
\]
whenever \eqref{F_h(p,q,s)-r} holds. 
This finishes the proof of part (ii), and thus of Theorem \ref{F-th3} as well.
\end{proof}

\vskip.01in
	\noindent{\bf Acknowledgements}.
 Z.-G. Wang was partially supported by the \textit{Key Project of Education Department of Hunan Province} under Grant no. 25A0668, and the \textit{Natural Science Foundation of Changsha} under Grant no. kq2502003 of the P. R. China.

\vskip.05in
	\noindent{\bf Conflicts of interest}. The authors declare that they have no conflict of interest.
	
\vskip.05in
\noindent{\bf Data availability statement}.  Data sharing is not applicable to this article as no datasets were generated or analysed during the current study.

\end{document}